\documentclass[]{article}
\usepackage{amssymb}
\usepackage{amsmath}
\usepackage{color}
\usepackage{float}
\usepackage{subfigure}
\usepackage{caption}
\usepackage{tikz}
\usepackage{hyperref}\hypersetup{hypertex=true,
            colorlinks=true,
            linkcolor=blue,
            anchorcolor=blue,
            citecolor=blue}

\usetikzlibrary{backgrounds}
\usetikzlibrary{arrows.meta}
\newtheorem{thm}{Theorem}[section]
\newtheorem{conj}{Conjecture}[section]
\newtheorem{coro}[thm]{Corollary}
\newtheorem{definition}{Definition}[section]
\newtheorem{lem}[thm]{Lemma}

\newcommand{\JBJ}[1]{{#1}}

\usepackage{amsfonts}
\newenvironment{proof}{{\medbreak\noindent\it  Proof.}\,}{\hfill$\square$\medbreak}

\usepackage{geometry}
\usetikzlibrary{shapes, arrows, calc, arrows.meta, fit, positioning} 
\tikzset{  
	-stealth,auto,node distance =1.5 cm and 1.3 cm, thick,
	state/.style ={circle, draw, inner sep=0.3pt}, 
	point/.style = {circle, draw, inner sep=0.18cm, fill, node contents={}},  
	bidirected/.style={Latex-Latex,dashed}, 
	el/.style = {inner sep=2.5pt, align=right, sloped}  
}  

\begin{document}


\title{Arc-disjoint out-branchings and in-branchings in semicomplete digraphs}

\author{J. Bang-Jensen\thanks{Department of Mathematics and Computer Science, University of Southern Denmark,Odense Denmark.(email: jbj@imada.sdu.dk)}\and Y. Wang \thanks{Department of Mathematics and Computer Science, University of Southern Denmark, Odense, Denmark and School of Mathematics, Shandong University, Jinan 250100, China. (email: yunwang@imada.sdu.dk,wangyun\_sdu@163.com)}}

\maketitle

\begin{abstract}
  An out-branching $B^+_u$ (in-branching $B^-_u$) in a digraph $D$  is a connected spanning subdigraph of $D$ in which every vertex except the vertex $u$, called the root,  has in-degree (out-degree)  one. It is well-known that there exists a polynomial algorithm for deciding whether a given digraph has $k$  arc-disjoint out-branchings with prescribed roots ($k$ is part of the input). In sharp contrast to this, it is already NP-complete to decide if a digraph has one  out-branching which is arc-disjoint from some in-branching. \JBJ{A digraph is {\bf semicomplete} if it has no pair of non adjacent vertices. A {\bf tournament} is a semicomplete digraph without directed cycles of length 2.} In this paper we give a complete classification of semicomplete digraphs which have an out-branching $B^+_u$ which is arc-disjoint from some in-branching $B^-_v$ where $u,v$ are prescribed vertices of $D$. Our characterization, which is surprisingly simple, generalizes a complicated characterization for tournaments from 1991 by the first author and our proof implies the existence of a polynomial algorithm for checking whether a given semicomplete digraph has such a pair of branchings for prescribed vertices $u,v$ and construct a solution if one exists. \JBJ{This confirms a conjecture of Bang-Jensen for the case of semicomplete digraphs.}

\end{abstract}

\noindent{\bf Keywords:} \texttt{arc-disjoint subdigraphs; in-branchings; out-branchings; semicomplete digraph; polynomial algorithm }

\section{Introduction}\label{intro}
Notation follows \cite{bang2009} so we only repeat a few definitions here (see also Section
\ref{sec:prelim}).
Let $D=(V,A)$ be a digraph. An {\bf out-tree} ({\bf in-tree}) is an oriented tree in which every vertex except one, called the {\bf root}, has in-degree (out-degree) one. An {\bf out-branching} ({\bf in-branching}) of $D$ is a spanning out-tree (in-tree) in $D$. For a subdigraph $H$ of $D$ and a vertex $s$ of $H$ we denote by $B_{s,H}^+$, (resp., $B_{s,H}^-$) an arbitrary out-branching (resp., in-branching) rooted at $s$ in $H$. To simplify the notation, we set $B_s^+=B_{s,D}^+$ and $B_s^-=B_{s,D}^-$. 

A digraph $D$ is {\bf strong} if there exists a path from $x$ to $y$  in $D$ for every ordered pair of distinct vertices $x$, $y$ of $D$ and $D$ is {\bf $k$-arc-strong} if $D\setminus{}A'$ is strong for every subset $A' \subseteq A$ of size at most $k - 1$. For a subset $X$ of $V$, we denote by $D\left\langle X \right\rangle$ the subdigraph of $D$ induced by $X$. 

The following well-known theorem, due to Edmonds, provides a characterization for the existence of $k$ arc-disjoint out-branchings rooted at the same vertex. 
\begin{thm}\label{edmonds1973} {\bf(Edmonds' Branching Theorem)} 
A directed multigraph $D = (V,A)$ with a special vertex $s$ has $k$ arc-disjoint out-branchings rooted at $s$ if and only if
\begin{equation}
  \label{kbranchcond}
  d^-(X) \geq k,\;\; \forall \;\; \emptyset\neq X \subseteq V - s.
  \end{equation}
\end{thm}

Note that, by Menger's Theorem, (\ref{kbranchcond}) is equivalent to the existence of $k$ arc-disjoint $(s,v)$-paths for every $v\in V-s$. Hence (\ref{kbranchcond}) can be checked in polynomial time via maximum flow calculations, see e.g.,
\cite[Section 5.4]{bang2009}.
Lov\'asz \cite{lovaszJCT21} gave a constructive proof of Theorem \ref{edmonds1973} which implies the existence of a polynomial algorithm for constructing a set of $k$ arc-disjoint branchings from a given root when (\ref{kbranchcond}) is satisfied. 

A natural related problem is to ask for a characterization of  digraphs having an out-branching and an in-branching which are arc-disjoint. Such pair will be called \textbf{a good pair} in this paper and more precisely we call it a \textbf{good $(u,v)$-pair} if the roots $u$ and $v$ are specified. Thomassen showed (see \cite{bangJCT51} and \cite{bangJGT100}) that for general digraphs it is NP-complete to decide if a given digraph has a good pair. This makes it interesting to study classes of digraphs for which we can find a good pair or decide that none exists in polynomial time.

For acyclic digraphs there can only be one choice for the vertices $u,v$ as $u$ must be able to reach all other vertices by a directed path and $v$ must be reachable by all other vertices by a directed paths. A characterization of acyclic digraphs with a good pair and a polynomial algorithm for finding a good pair when it exists was given in \cite{bangJGT42}. A polynomial algorithm was also given in
\cite{bercziIPL109}. In \cite{bangJCT51} 
the first author gave a complete characterization of tournaments with no good $(u,v)$-pair \JBJ{and gave a polynomial algorithm for either producing a good $(u,v)$-pair  for  a given tournament $T$ and two vertices $u,v$ of $T$ or providing a certificate for the non-existence of such a pair in $T$. } Bang-Jensen and Huang characterized quasi-transitive digraphs with a good $(u,u)$-pair \cite{bangJGT20b}. A digraph is {\bf quasi-transitive} if the presence of arcs $uv$ and $vw$ implies an arc between $u$ and $w$. It is easy to see that every semicomplete digraph is quasi-transitive. Gutin and Sun  \cite{gutinDM343} generalized this result to digraphs of the form $D=T[H_1,H_2,\ldots{},H_{|V(T)|}]$, where $T$ is a semicomplete digraph. Such a digraph is called a {\bf composition of $T$} and the precise definition is not important here (see e.g. \cite[Page 9]{bang2009}).\\

\JBJ{The following conjecture, due to Thomassen, is wide open and it is not even known whether already $K=3$ suffices for all digraphs.}

\begin{conj}\cite{thomassen1989}\label{conj:CT}
  There exists an integer $K$ such that every $K$-arc-strong digraph $D$ has a good $(u,v)$-pair for every choice of vertices $u,v$ of $D$.
  \end{conj}
Bang-Jensen, Bessy, Havet and Yeo \cite{bangJGT100} showed that every digraph of independence number at most 2 and arc-connectivity at least 2 has a good $(u,v)$-pair \JBJ{for at least one choice of vertices $u,v$} and they showed  that the same condition is not sufficient to guarantee a good $(u,v)$-pair for every choice of $u$ and $v$. Hence $K$  in Conjecture \ref{conj:CT} must be at least 3. To the best of our knowledge it is open whether $K=3$ would suffice for all digraphs.

The following conjecture due to Bang-Jensen and Yeo would imply Conjecture \ref{conj:CT}.
\begin{conj}\cite{bangC24}
  \label{conj:BJAY}
  There exists an integer $K$ such that every $K$-arc-strong digraph $D=(V,A)$ has an arc-partition $A=A_1\cup A_2$ such that each of the subdigraphs $D_1=(V,A_1)$ and $D_2=(V,A_2)$ are spanning and strong.
\end{conj}

The next result implies that the Conjecture \ref{conj:BJAY} holds with $K=3$  for the case of semicomplete digraphs.

\begin{thm}\cite{bangC24}
  \label{thm:SD2arcstrong}
  Let $D=(V,A)$ be a 2-arc-strong semicomplete digraph. Then $D$ has an arc-partition $A=A_1\cup A_2$ such that each of the subdigraphs $D_1=(V,A_1)$ and $D_2=(V,A_2)$ are spanning and strong except if $D$ is isomorphic to the digraph $S_4$ on vertices $\{v_1,v_2,v_3,v_4\}$ and
  arc set \\
  $\{v_1v_2,v_2v_3,v_3v_4,v_4v_1,v_2v_4,v_4v_2,v_1v_3,v_3v_1\}$.
\end{thm}

It is not difficult to check that $S_4$ has a good $(u,v)$-pair for all possible choices of $u,v\in V(S_4)$. Hence every 2-arc-strong semicomplete digraph has a good pair for every possible choice of $u,v$. \JBJ{So Conjecture \ref{conj:CT} holds for semicomplete digraphs with $K=2$.} In \cite{bangJGT95} Theorem \ref{thm:SD2arcstrong} was generalized to semicomplete compositions, that is, digraphs of the form $S[H_1,\ldots{},H_{|V(S)|}]$. From that result a complete characterization for the existence of good $(u,v)$-pairs in  2-arc-strong semicomplete compositions can be obtained.

Bang-Jensen and Huang considered the case that $u=v$ for strong semicomplete digraphs by proving Theorem \ref{sABC}.

\begin{thm}\cite{bangJGT20b} \label{sABC}
  Let $D$ be a strong semicomplete digraph and let $u\in V(D)$ be arbitrary vertex. Suppose that $D$ does not contain a good pair with the same root $u$. Then the following holds, where $A, B, C$ form a partition of
  $V(D)-u$ such that $N^+_D(u) = A\cup C$ and $N^-_D(u)= B\cup C$. 
  There is precisely one arc $e$ leaving the terminal component of $D\left\langle A \right\rangle$ and precisely one arc $e^{\prime}$ entering the initial component of $D\left\langle B\right\rangle$ and $e=e^{\prime}$.
\end{thm}

\JBJ{For two vertices $x,y$ of a digraph $D$ we} use $P_{x,y}$ to denote a path from $x$ to $y$. Such a path is also called an $(x,y)$-path.
Bang-Jensen's characterization of tournaments with good $(u,v)$-pairs in \cite{bangJCT51} is quite complicated and does not extend  to semicomplete digraphs \JBJ{so we will not describe it here}.
In this paper, we prove the following  a surprisingly simple characterization for the existence of a good $(u,v)$-pair in semicomplete digraphs.

\begin{figure}[H]
\centering
\subfigure{\begin{minipage}[t]{0.15\linewidth}
\centering\begin{tikzpicture}[scale=0.8]
		\filldraw[black](0,9) circle (3pt)node[label=left:$u$](u){};
		\filldraw[black](0,8) circle (3pt)node[label=left:$v$](v){};
		\path[draw, line width=0.8pt] (u) edge (v);
	\end{tikzpicture}\caption*{(a)}\end{minipage}}
\subfigure{\begin{minipage}[t]{0.15\linewidth}
\centering\begin{tikzpicture}[scale=0.8]
		\filldraw[black](0,10) circle (3pt)node[label=left:$u$](u){};
		\filldraw[black](0,9) circle (3pt)node[](w){};
		\filldraw[black](0,8) circle (3pt)node[label=left:$v$](v){};
			
		\path[draw, line width=0.8pt] (u) edge[bend left=30] (v);
		\path[draw, line width=0.8pt] (u) edge (w) edge (v);
		\end{tikzpicture}\caption*{(b)}\end{minipage}}
	\subfigure{\begin{minipage}[t]{0.15\linewidth}
			\centering\begin{tikzpicture}[scale=0.8]
		\filldraw[black](0,10) circle (3pt)node[label=left:$u$](u){};
		\filldraw[black](0,9) circle (3pt)node[](w){};
		\filldraw[black](0,8) circle (3pt)node[label=left:$v$](v){};
		
		\path[draw, line width=0.8pt] (u) edge[bend left=30] (v);
		\path[draw, line width=0.8pt]  (v) edge[bend left=30] (u);
		\path[draw, line width=0.8pt] (u) edge (w) edge (v);
	\end{tikzpicture}\caption*{(c)}\end{minipage}}
\subfigure{\begin{minipage}[t]{0.15\linewidth}
		\centering\begin{tikzpicture}[scale=0.8]
		\filldraw[black](0,10) circle (3pt)node[label=left:$u$](a){};
		\filldraw[black](0,9) circle (3pt)node[](b){};
		\filldraw[black](0,8) circle (3pt)node[](c){};
		\filldraw[black](0,7) circle (3pt)node[label=left:$v$](d){};
			
		\path[draw, line width=0.8pt] (a) edge[bend left=30] (c);
		\path[draw, line width=0.8pt] (b) edge[bend left=30] (d);
		\path[draw, line width=0.8pt] (d) edge[bend right=50] (a);
		\path[draw, line width=0.8pt] (a) edge (b) edge (c) edge (d);
		\end{tikzpicture}\caption*{(d)}\end{minipage}}
\subfigure{\begin{minipage}[t]{0.15\linewidth}
	\centering\begin{tikzpicture}[scale=0.8]
		\filldraw[black](0,10) circle (3pt)node[label=left:$u$](a){};
	    \filldraw[black](0,9) circle (3pt)node[](b){};
		\filldraw[black](0,8) circle (3pt)node[](c){};
		\filldraw[black](0,7) circle (3pt)node[label=left:$v$](d){};
					
		\path[draw, line width=0.8pt] (a) edge[bend left=30] (c);
		\path[draw, line width=0.8pt] (b) edge[bend left=30] (d);
		\path[draw, line width=0.8pt] (d) edge[bend left=30] (b);
		\path[draw, line width=0.8pt] (d) edge[bend right=50] (a);
		\path[draw, line width=0.8pt] (a) edge (b) edge (c) edge (d);
		\end{tikzpicture}\caption*{(e)}\end{minipage}}
\subfigure{\begin{minipage}[t]{0.15\linewidth}
		\centering\begin{tikzpicture}[scale=0.8]
		\filldraw[black](0,10) circle (3pt)node[label=left:$u$](a){};
		\filldraw[black](0,9) circle (3pt)node[](b){};
		\filldraw[black](0,8) circle (3pt)node[](c){};
		\filldraw[black](0,7) circle (3pt)node[label=left:$v$](d){};
		
		\path[draw, line width=0.8pt] (a) edge[bend left=30] (c);
		\path[draw, line width=0.8pt] (b) edge[bend left=30] (d);
		\path[draw, line width=0.8pt] (c) edge[bend left=30] (a);
		\path[draw, line width=0.8pt] (d) edge[bend right=50] (a);
		\path[draw, line width=0.8pt] (a) edge (b) edge (c) edge (d);
		\end{tikzpicture}\caption*{(f)}\end{minipage}}
\caption{Semicomplete digraphs that have no good $(u,v)$-pair. The digraph in (e) is isomorphic to the digraph in (f).}
\label{fig4}
\end{figure}

\begin{thm}
  \label{thm:SDbranchchar}
  Let $D=(V,A)$ be a semicomplete digraph  with $u,v\in V$ (possibly $u=v$). Then $D$ has a good $(u,v)$-pair if and only if it satisfies (i) and (ii) below.
  \begin{itemize}
  \item[(i)] For every choice of $z,w\in V$ there are arc-disjoint paths $P_{u,z},P_{w,v}$ in $D$
  \item[(ii)] $D$ is not one of the digraphs in Figure  \ref{fig4}(b)-(f).
    \end{itemize}
  \end{thm}

  It is easy to see that (i) must hold if $D$ has a good $(u,v)$-pair
  and when $D$ has at least 5 vertices the theorem says that (i) is also sufficient.

{\bf The rest of the paper is organized as follows:} We start out with Section \ref{sec:prelim} which contains some extra definitions and results that will be used in the paper. In Section \ref{sec:paths} we unify  results from \cite{bangJCT51}  on arc-disjoint paths in semicomplete digraphs in order to use these when we assume condition (i) of Theorem \ref{thm:SDbranchchar} holds. In Section \ref{sec:extend} we prove some important lemmas on how to extend special arc-disjoint in- and out-trees to larger ones. Then in Section \ref{sec:B+P} we study the problem of finding an out-branching rooted at a specific vertex which is arc-disjoint from a path with prescribed end vertices. Finally we prove Theorem \ref{thm:SDbranchchar} in Section \ref{sec:proof}. We also give an alternative, semingly more involved characterization of semicomplete digraphs without a good $(u,v)$-pair for specified vertices $u,v$. This characterization is used heavily in \cite{bangQTDbranchpaper}.

\section{Further terminology and Preliminaries}\label{sec:prelim}

If a digraph is not strong, then we can label its strong components $D_1,\ldots{},D_p$ such that there is no arc from $D_j$ to $D_i$ when $j>i$. We call such an ordering an \textbf{acyclic ordering} of the strong components of $D$. For a non-strong semicomplete digraph $D$  it is easy to see that the ordering  $D_1,\ldots{},D_p$ is unique and we call $D_1$ (resp., $D_p$) the {\bf initial (resp., terminal)} strong component of $D$.

The set of vertices of a digraph $D=(V,A)$ which can reach (resp., can be reached from) every other vertex in $V$ by a directed path is denoted by $Out(D)$ (resp., $In(D)$).  
Note that  a vertex $v$ belongs to $Out(D)$ (resp., $In(D)$) if and only if it is the root of some out-branching (resp., in-branching) of $D$. Hence  $Out(D)=In(D)$ if and only if $D$ is strong.

\begin{lem}\label{outinstrong}\cite{bangJGT100}
Let $D$ be a  digraph. Then the induced subdigraphs $D\left\langle Out(D)\right\rangle$ and
$D\left\langle In(D)\right\rangle$ are strong.
\end{lem}

We will use the following classical result by Camion. It was originally proved only for tournaments but almost the same proof works for semicomplete digraphs (one can also use the easy fact that every strong semicomplete digraph $D=(V,A)$ on at least 3 vertices contains a spanning strong subtournament $T=(V,A')$ where $A'\subseteq A$).

\begin{thm}\cite{camionCRASP249}\label{thm:camion}
Every strong semicomplete digraph of order at least 2 has a hamiltonian cycle.
\end{thm}

The following extension of Redei's Theorem \cite{redeiALS7} is easy to prove from Theorem \ref{thm:camion} and the fact that if $D$ is not strong and $D_1,\ldots{},D_p$, $p\geq 2$ is the unique ordering of the strong components of $D$ then every vertex of $V(D_i)$ dominates every vertex of $V(D_j)$ when $1\leq i<j\leq p$.
\begin{lem}
  Every strong semicomplete digraph $D$  has a hamiltonian path starting at any prescribed vertex $x$.
  \JBJ{If $D$ is a non-strong semicomplete digraph and  $D_1,\ldots{},D_p$, $p\geq 2$ is the unique ordering of its strong components}, then $D$  has an $(x,y)$-hamiltonian path for every choice of vertices $x\in V(D_1)$ and $y\in V(D_p)$.
  \end{lem}

\begin{thm}\cite{bangJGT100}
Let $D$ be a semicomplete digraph of order at least 4 and let $v$ be arbitrary vertex of $In(D)$. There is a pair of arc-disjoint out- and in-branchings in $D$ such that the in-branching is rooted at $v$ if and only if $D$ is not a digraph such that both $v$ and its in-neighbor has in-degree one.
\end{thm}

\begin{lem}\cite{bangJGT100}\label{nonstrong-OI}
Every non-strong semicomplete digraph of order at least 4 has a good $(u, v)$-pair for every choice of $u \in Out(D)$ and $v \in In(D)$.
\end{lem}

\begin{lem}\label{OutbranPath}
Let $D$ be a digraph and let $u,v$ be two vertices of $D$ such that \JBJ{$u\in Out(D)$.} Suppose that $D$ has no out-branching  $B_u^+$ which is arc-disjoint from some $(u,v)$-path. Then there exists a partition $V_1,V_2$ of $V(D)$ such that $v\in V_1,u\in V_2$ and $d_D^+(V_2)=1$.
\end{lem}

\begin{proof}
By the assumption, there are no two arc-disjoint out-branchings rooted at $u$ in $D$. It follows by Edmonds' branching theorem (Theorem \ref{edmonds1973}) that there exists a partition $U_1,U_2$ of $V(D)$ such that $u\in U_2$ and $d_D^-(U_1)=d_D^+(U_2)\leq 1$. Since $D$ contains an out-branching with root $u$ and $u\in U_2$, we have $d_D^+(U_2)=1$, moreover, let $xy$ be the only arc leaving $U_2$, then $y$ dominates all vertices of $U_2-x$ and it can reach all vertices of $U_1$.

If $v\in U_1$, then $U_1,U_2$ is the desired partition. So we may assume that $v$ belongs to $U_2$. By the assumption of the lemma, there can be  no pair of arc-disjoint $(u,x)$- and $(u,v)$-paths in $D\left\langle  U_2\right\rangle$. Otherwise, using that $y$ dominates all of $U_2-x$, \JBJ{it is easily seen that $D$ has an out-branching $B_u^+$ which is arc-disjoint from some $(u,v)$-path}, contradiction. Thus by Menger's theorem, there exists a proper subset $U$ of $U_2$ such that $u\in U,\{v,x\} \subseteq U_2- U$ and $d_D^+(U)=1$. Then $V_2=U$ and $V_1=V(D)-U$ is the desired partition.
\end{proof}

\section{Arc-disjoint paths in semicomplete digraphs}\label{sec:paths}

In this section we unify and slightly generalize results on arc-disjoint paths from \cite{bangJCT51} in order to use these in the next sections.

\begin{definition}\label{def1}
Let $D=(V,A)$ be a semicomplete digraph and let $u,w,v$ be three vertices of $D$. The 4-tuple $(D,u,w,v)$ is said to be of 

{\bf type $A$}, if there exists a partition $V_1, V_2$ of $V$ such that $v\in V_1$, $u,w\in V_2$ and there is exactly one arc from $V_2$ to $V_1$. (Note that in this type $D$ may be non-strong.)

{\bf type $B$}, if there exists a partition $V_1, V_2, V_3$ of $V$ such that $u,v\in V_2$, $w\in V_3$ and all arcs between $V_i$ and $V_j$ with $i<j$ \JBJ{go} from $V_i$ to $V_j$ \JBJ{except for precisely one arc}  which goes from the terminal component of $D\left\langle  V_3\right\rangle$ to the initial component of $D\left\langle V_1\right\rangle$.

{\bf type $2\alpha+2$}, for some $\alpha\geq 1$ if there exists a partition $V_1,\ldots,V_{2\alpha+2}$ of $V$ such that $v\in V_2, w\in V_{2\alpha+1}, u\in V_{2\alpha+2}$ and all arcs between $V_i$ and $V_j$ with $i<j$ \JBJ{go} from $V_i$ to $V_j$ with the following exceptions. There exists precisely one arc from $V_{i+2}$ to $V_{i}$ for all $i\in[2\alpha]$ and it goes from the terminal component of $D\left\langle  V_{i+2}\right\rangle$ to the initial component of $D\left\langle  V_i\right\rangle$.

{\bf type $2\alpha+3$}, for some $\alpha\geq 1$ if there exists a partition $V_1,\ldots,V_{2\alpha+3}$ of $V$ such that $v\in V_2, u\in V_{2\alpha+2}, w\in V_{2\alpha+3}$ and all arcs between $V_i$ and $V_j$ with $i<j$ \JBJ{go} from $V_i$ to $V_j$ with the following exceptions. There exists precisely one arc from $V_{i+2}$ to $V_{i}$ for all $i\in[2\alpha+1]$  and it goes from the terminal component of $D\left\langle  V_{i+2}\right\rangle$ to the initial component of $D\left\langle  V_i\right\rangle$.
\end{definition} 

\begin{figure}[H]
	\centering
	\subfigure{
		\begin{minipage}[t]{0.2\linewidth}
			\centering\begin{tikzpicture}[scale=0.8]
			\draw[line width=1pt] (0,9) ellipse [x radius=20pt, y radius=10pt];
			\draw[line width=1pt] (0,7.3) ellipse [x radius=20pt, y radius=10pt];
			\draw[-stealth,line width=1.5pt] (-1.1,8.5) -- (-1.1,7.5); 
			\filldraw[white](0.5,7.3) circle (3pt)node[](a){}; \filldraw[white](0.5,9) circle (3pt)node[](b){};
			\path[draw, line width=1pt] (a) edge[bend right=40] (b);
			\coordinate [label=center:$u\;w$] () at (0,7.3);
			\coordinate [label=center:$v$] () at (0,9);
			\end{tikzpicture}\caption*{Type $A$}
		\end{minipage}}
	\subfigure{
		\begin{minipage}[t]{0.2\linewidth}
			\centering\begin{tikzpicture}[scale=0.8]
			\draw[line width=1pt] (0,10) ellipse [x radius=20pt, y radius=10pt];
			\draw[line width=1pt] (0,8.6) ellipse [x radius=20pt, y radius=10pt];
			\draw[line width=1pt] (0,7.3) ellipse [x radius=20pt, y radius=10pt];
			\draw[-stealth,line width=1.5pt] (-1.1,9.5) -- (-1.1,7.5); 
			\filldraw[white](0.5,7.3) circle (3pt)node[](a){}; \filldraw[white](0.5,10) circle (3pt)node[](b){};
			\path[draw, line width=1pt] (a) edge[bend right=40] (b);
			\coordinate [label=center:$u\;v$] () at (0,8.6);
			\coordinate [label=center:$w$] () at (0,7.3);
			\end{tikzpicture}\caption*{Type $B$}
		\end{minipage}}
	\subfigure{\begin{minipage}[t]{0.2\linewidth}
			\centering\begin{tikzpicture}[scale=0.8]
			\draw[line width=1pt] (0,10) ellipse [x radius=20pt, y radius=10pt];
			\draw[line width=1pt] (0,9) ellipse [x radius=20pt, y radius=10pt];
			\draw[line width=1pt] (0,8) ellipse [x radius=20pt, y radius=10pt];
			\draw[line width=1pt] (0,7) ellipse [x radius=20pt, y radius=10pt];
			\draw[-stealth,line width=1.5pt] (-1.1,9.5) -- (-1.1,7.5); 
			\filldraw[white](0.5,7.3) circle (0.1pt)node[](b){};
			\filldraw[white](0.5,8.3) circle (0.1pt)node[](c){}; 
			\filldraw[white](0.5,9.3) circle (0.1pt)node[](d){};
			\filldraw[white](0.5,10.3) circle (0.1pt)node[](e){}; 
			\filldraw[white](0.5,6.7) circle (0.1pt)node[](b1){};
			\filldraw[white](0.5,7.7) circle (0.1pt)node[](c1){};
			\filldraw[white](0.5,8.7) circle (0.1pt)node[](d1){};
			\path[draw, line width=0.8pt] (b1) edge[bend right=60] (d);
			\path[draw, line width=0.8pt] (c1) edge[bend right=60] (e);
			\coordinate [label=center:$u$] () at (0,7);
			\coordinate [label=center:$v$] () at (0,9);
			\coordinate [label=center:$w$] () at (0,8);
			\end{tikzpicture}\caption*{Type $2\alpha+2$}
		\end{minipage}}
	\subfigure{\begin{minipage}[t]{0.2\linewidth}
			\centering\begin{tikzpicture}[scale=0.8]
			\draw[line width=1pt] (0,10) ellipse [x radius=20pt, y radius=10pt];
			\draw[line width=1pt] (0,9) ellipse [x radius=20pt, y radius=10pt];
			\draw[line width=1pt] (0,8) ellipse [x radius=20pt, y radius=10pt];
			\draw[line width=1pt] (0,7) ellipse [x radius=20pt, y radius=10pt];
			\draw[line width=1pt] (0,6) ellipse [x radius=20pt, y radius=10pt];
			\draw[-stealth,line width=1.5pt] (-1.1,9.5) -- (-1.1,6.5); 
			\filldraw[white](0.5,7.3) circle (0.1pt)node[](b){};
			\filldraw[white](0.5,8.3) circle (0.1pt)node[](c){}; 
			\filldraw[white](0.5,9.3) circle (0.1pt)node[](d){};
			\filldraw[white](0.5,10.3) circle (0.1pt)node[](e){}; 
			\filldraw[white](0.5,5.7) circle (0.1pt)node[](a1){}; 
			\filldraw[white](0.5,6.7) circle (0.1pt)node[](b1){};
			\filldraw[white](0.5,7.7) circle (0.1pt)node[](c1){}; 
			\filldraw[white](0.5,8.7) circle (0.1pt)node[](d1){};
			\path[draw, line width=0.8pt] (a1) edge[bend right=60] (c);
			\path[draw, line width=0.8pt] (b1) edge[bend right=60] (d);
			\path[draw, line width=0.8pt] (c1) edge[bend right=60] (e);
			\coordinate [label=center:$u$] () at (0,7);
			\coordinate [label=center:$v$] () at (0,9);
			\coordinate [label=center:$w$] () at (0,6);
			\end{tikzpicture}\caption*{Type $2\alpha+3$}
		\end{minipage}}
	\caption{\JBJ{Illustration of Definition \ref{def1}}. The vertex sets \JBJ{$V_1,V_2\ldots{}$} are labeled from top to bottom. The bold arcs indicate that all  arcs not shown going up in the figure are present in the shown direction. The third and fourth digraphs are of type $2\alpha+2$ and $2\alpha+3$ with $\alpha=1$, respectively.}
	\label{fig1}
\end{figure}

\begin{lem}\label{type--nopaths}
  Let $D$ be a semicomplete digraph let $u,w,v$ be three vertices of $D$. Suppose that $(D,u,w,v)$ is one of the types in Definition \ref{def1}, then for any vertex $z\in V_1$, there  is no pair of  arc-disjoint $(u,z)$- and $(w,v)$-paths in $D$.
\end{lem}
\begin{proof}
	It is not difficult to check that if $(D,u,w,v)$ is of the type $A$ or $B$, then $D$ cannot contain arc-disjoint $(u,z)$- and $(w,v)$-paths as both of the paths must use the arc entering $V_1$.
	
	For the case that $(D,u,w,v)$ is of type $2\alpha+2$ for some $\alpha\geq 1$, we use $x_{i+2}y_i$ to denote the arc from the terminal component of $D\left\langle  V_{i+2}\right\rangle$ to the initial component of $D\left\langle  V_i\right\rangle$. Let $P$ be an arbitrary $(u,z)$-path. Note that the path $P$ must use the arc $x_{2\alpha+2}y_{2\alpha}$ and at least one of the arcs of kind $x_{2k+1}y_{2k-1},~k\in[\alpha]$. Let $x_{2j+1}y_{2j-1},~j\in[\alpha]$ be the first arc of the kind $x_{2k+1}y_{2k-1}$ as we go along $P$ from $u$. In $D-A(P)$ there is no path from $w$ to $\bigcup_{i<2j+1}V_i$, because there are only two arcs, i.e., $x_{2j+2}y_{2j}$ and $x_{2j+1}y_{2j-1}$, entering $\bigcup_{i<2j+1}V_i$ and these two arcs are in $A(P)$. Note that $v\in V_2 \subseteq  \bigcup_{i<2j+1}V_i$ as $j\in[\alpha]$, so there is no $(w,v)$-path in $D-A(P)$ and then $D$ has no arc-disjoint $(u,z)$- and $(w,v)$-paths.
	
	Similarly, it is not difficult to prove that if $(D,u,w,v)$ is of type $2\alpha+3$ for some $\alpha\geq 1$, there cannot exist arc-disjoint $(u,z)$- and $(w,v)$-paths in $D$.
      \end{proof}

      The next two results from \cite{bangJCT51} were only stated and proved for tournaments but it is easy to check that the proofs are also valid for semicomplete digraphs.

\begin{thm}\label{jbjuzvw}\cite{bangJCT51}
Let $D$ be a semicomplete digraph and let $x_1,y_1,x_2,y_2$ be distinct vertices such that $D$ contains an $(x_i,y_i)$-path for all $i\in[2]$. Then $D$ has a pair of arc-disjoint $(x_1,y_1)$- and  $(x_2,y_2)$-paths unless $x_1,y_1,x_2,y_2$ all belong to the same strong component $D_j$ of $D$ and for some $i\in[2]$,  $(D_j,x_i,x_{3-i},y_{3-i})$ is one of the types in Definition \ref{def1} and the vertex $y_i$ belongs to $V_1$.
\end{thm}

\begin{thm}\label{jbjuvw}\cite{bangJCT51}
Let $D$ be a semicomplete digraph and $x,y,z\in V(D)$ three distinct vertices such that there exist an $(x,y)$- and a $(y,z)$-path in $D$. There exists a pair of arc-disjoint $(x,y)$- and $(y,z)$-paths in $D$ if and only if for every arc $e$ there exists either an $(x,y)$-path or a $(y,z)$-path in $D-e$.
\end{thm}

Now we \JBJ{prove} the following common generalization of Theorems \ref{jbjuzvw} and \ref{jbjuvw}.
\begin{thm}\label{goodpair-paths}
Let $D$ be a semicomplete digraph and let $x_1,y_2,x_2,y_2$ be four vertices (not necessarily distinct) such that $D$ contains an $(x_i,y_i)$-path for all $i\in[2]$. Then $D$ has a pair of arc-disjoint $(x_1,y_1)$- and $(x_2,y_2)$-paths unless one of the following statements holds.
		
	(i) $D$ is non-strong, $x_1=x_2,y_1=y_2$ and $\{x_1\}, \{y_1\}$ are two consecutive components  in the acyclic ordering of the strong components of $D$.
	
	(ii) The four vertices all belong to the same strong component $D_j$ and for some $i$, $(D_j,x_i,x_{3-i},y_{3-i})$ is one of the types in Definition \ref{def1} and the vertex $y_i$ belongs to $V_1$.
\end{thm}
\begin{proof} 
	There is clearly no arc-disjoint $(x_1,y_1)$- and $(x_2,y_2)$-paths in $D$ when (i) holds. If (ii) holds, by Lemma \ref{type--nopaths}, there is no arc-disjoint $(x_1,y_1)$- and $(x_2,y_2)$-paths in $D_j$ and consequently no such pair in $D$.
	
	Next we suppose that none of (i) and (ii) holds and show that $D$ has a pair of arc-disjoint $(x_1,y_1)$- and $(x_2,y_2)$-paths. Suppose first that $x_1,y_2,x_2,y_2$ do not all belong to the same strong component. In particular  $D$ is non-strong. Let $D_1,\ldots,D_l~(l\geq 2)$ be  the unique acyclic ordering of the strong components of $D$. If $x_i$ and $y_i$ belong to the same component of $D$ for all $i\in[2]$, then $D$ clearly has the desired   paths. So we may assume that $x_1$ and $y_1$ belong to components $D_p$ and $D_q$ with $q>p$, respectively. Then $x_1y_1\in A(D)$ and $D$ has the desired paths when $x_2\notin D_p$ or $y_2\notin D_q$. Therefore we may assume that $x_1,x_2\in D_p$ and $y_1,y_2\in D_q$. Let $D^{\prime}=D\left\langle V(D_p\cup\cdots\cup D_q)\right\rangle$. If there is a good $(x_1,y_2)$-pair in $D^{\prime}$, then $D^{\prime}$ (and consequently, $D$) has the desired paths. We may assume that $|D^{\prime}|\leq 3$ by Lemma \ref{nonstrong-OI}. It is not difficult to check that $D$ has the desired paths if $|D^{\prime}|= 3$. Hence, $|D^{\prime}|=2$, which  implies that (i) holds, a contradiction with our assumption.

	Therefore we may assume  that $x_1,y_2,x_2,y_2$ all belong to the same strong component $D_j$. By Theorem \ref{jbjuzvw}, we may assume that $|\{x_1\}\cup\{y_1\}\cup\{x_2\}\cup\{y_2\}|<4$. By the assumption in the lemma $D_j$ has an $(x_i,y_i)$-path for all $i\in[2]$, so we can assume that $x_i\neq y_i$ for all $i\in[2]$. If $y_1=y_2$, then it follows from by Menger's theorem and the fact that $(D_j,x_1,x_2,y_2)$ is not of the type $A$ that the desired paths exist. So we may assume that $y_1\neq y_2$ and  by symmetry we have $x_1\neq x_2$. 

        The only remaining case is $x_i=y_{3-i}$ for some $i\in [2]$. Assume without loss of generality that
        $x_1=y_2$. If $D_j$ has a good $(x_1,x_1)$-pair, then $D_j$ clearly has the desired paths. So we may assume that $D_j$ has the structure given in Theorem \ref{sABC}. Let $e=e^{\prime}=ab$ and let $A_1,\ldots, A_l$ be the acyclic ordering of the strong components of $D\left\langle A\right\rangle$. Clearly, $\{x_1r:r\in A\cup C\}\cup \{ab\}\cup B_{b,B}^+$ is an out-branching $B^+_{x_1}$ with root $x_1$ in $D_j$, where $B_{b,B}^+$ exists as $b\in Out(B)$. If $D_j-A(B^+_{x_1})$ has an  $(x_2,x_1)$-path, then we are done so we may assume that $x_2$ can not reach any vertex of $\{x_1\}\cup B\cup C$ in $D_j-A(B^+_{x_1})$. This means that $x_2\in A_t$ for some $t$. Let $V_3$ be the set of vertices which $x_2$ can reach in $D_j-A(B^+_{x_1})$. Clearly, $x_2\in\bigcup_{i=t}^l A_i\subseteq V_3\subseteq A$ and there is only one arc $ab$ leaving $V_3$. By symmetry, one can construct an in-branching $B^-_{x_1}$ rooted at $x_1$ in $D_j$ and let $V_1$ be the set of vertices which can reach $y_1$ in $D_j-A(B^-_{x_1})$. Then $y_1\in V_1\subseteq B$ and there is only one arc $ab$ entering $V_1$. Set $V_2=V-V_1-V_3$. Then $(D_j,x_1,x_2,y_2)$ is of type $B$ with partition $V_1,V_2,V_3$ and $y_1\in V_1$, which contradicts  our assumption. 
	 This completes the proof. 
\end{proof}

\section{Extending arc-disjoint in- and out-trees in semicomplete digraphs}\label{sec:extend}
\begin{lem}\label{adjacentsomevex}
Let $D$ be a semicomplete digraph and let $H\subseteq D$ be a subdigraph. For any oriented tree $T$ in $D$, if all arcs in $H$ not in $A(T)$ are adjacent to some (fixed) vertex $h$ of $H$, then the digraph $H-h$ is either a single vertex or two vertices joined by one arc. In particular, $|V(H)|\leq 3$.
\end{lem}
\begin{proof}
  The lemma follows by the fact that $H-h$ is semicomplete and all arcs in $H-h$ are used in $T$.
  Hence $H-h$ has at most two vertices and if it has such vertices $u,v$ then there is only one arc between these.
\end{proof}

Let $D$ be a semicomplete digraph and let $X,Y$ be two disjoint subsets of $V(D)$ such that all vertices in $X$ (resp., in $Y$) are covered by an out-tree $T^+_u$ rooted at $u$ (resp., an in-tree $T^-_v$ rooted at $v$) in $D$ and  assume that $T^+_u$ and $T^-_v$ are arc-disjoint. Let $X^{\prime}\subseteq X$ (resp., $Y^{\prime}\subseteq Y$) be the set of vertices covered by $T^-_v$ (resp., $T^+_u$) (possibly $X^{\prime}$ or $Y^{\prime}$ is empty).

We say that the pair $(T^+_u,T^-_v)$ is \textbf{extendable} (on $X\cup Y$) if $D$ has an out-tree $\hat{T}^+_u$  and an in-tree $\hat{T}^-_v$  which are arc-disjoint and such that 
 each of them covers all vertices in $X\cup Y$. \JBJ{Note that it is not required that all arcs of $T^+_u$ ($T^-_v$) are arcs of $\hat{T}^+_u$  ($\hat{T}^-_v$).} The following lemmas are used to describe non-extendable pairs $(T^+_u,T^-_v)$. \JBJ{Note that $X$ and $Y$ are always the same below.}

\begin{lem}\label{XYnoarc}
Suppose that $X$ dominates $Y$ and no  arc between $X$ and $Y$ is used in $T^+_u$ or $T^-_v$. If $(T^+_u,T^-_v)$ is  non-extendable, then $X^{\prime}=Y^{\prime}=\emptyset$, $|X\cup Y|\leq 3$ and $D\left\langle X\right\rangle$ (resp., $D\left\langle Y\right\rangle$) is a single vertex or an arc covered by $T^+_u$ (resp., by $T^-_v$). 
\end{lem}
\begin{proof}
Suppose that $X^{\prime}\neq \emptyset$ or there is an arc $ab$ in $D\left\langle X\right\rangle$ not used in $T^+_u$, let $x\in X^{\prime}\cup \{a\}$, then $(T^+_u,T^-_v)$ can be extended as follows: $T^+_u\cup \{xr:r\in Y-Y^{\prime}\}$, $T^-_v\cup \{ry:r\in X-X^{\prime}\}$ or $T^-_v\cup \{ab\}\cup \{ry:r\in X-a\}$, where $y$ is any vertex of $Y$. By our assumption we may assume that $X^{\prime}=\emptyset$ and all arcs in $D\left\langle X\right\rangle$ are used in $T^+_u$. By symmetry, we have $Y^{\prime}=\emptyset$ and all arcs in $D\left\langle Y\right\rangle$ are used in $T^-_v$. This means that $|X|\leq 2$ and $|Y|\leq 2$. If $|X|=|Y|=2$, say $X=\{x_1,x_2\},Y=\{y_1,y_2\}$, then $(T^+_u,T^-_v)$ can be extended by adding arcs $x_iy_i,i\in[2]$ to $T^+_u$ and arcs $x_iy_{3-i},i\in[2]$ to $T^-_v$. So the lemma holds by assumption.
\end{proof}

\begin{lem}\label{XYblue}
Let $a\in Y, b\in X$ be two vertices. Suppose that the arc between $a$ and $b$ belongs to $T^-_v$ and all other arcs between $X$ and $Y$ go from $X$ to $Y$ and none of these arcs are  used in $T^+_u$ and $T^-_v$. If $(T^+_u,T^-_v)$ is non-extendable, then $X^{\prime}=\{b\}$ and one of the following statements holds:

(i) $X=X^{\prime}$, $a\notin Y^{\prime}$ and all in-arcs of $a$ in $D\left\langle X\cup Y\right\rangle$ are used in $T^-_v$.

(ii) $|(X-b)\cup Y|\leq 3$ and $D\left\langle X-b\right\rangle$ (resp., $D\left\langle Y\right\rangle$) is a single vertex or an arc covered by $T^+_u$ (resp., by $T^-_v$). Moreover, all arcs in $D\left\langle X\right\rangle$ which not covered by $T^+_u$ are out-arcs of $b$.	 
\end{lem}
\begin{proof}
As the arc between $a$ and $b$ is used in $T^-_v$, the vertex $b$ clearly belongs to $X^{\prime}$. If there exists an $x\in X^{\prime}-b$, then $(T^+_u,T^-_v)$ can be extended as follows: $T^+_u\cup \{xr:r\in Y-Y^{\prime}\}$ and $T^-_v\cup \{ra:r\in X-X^{\prime}\}$. It follows by the assumption that $x$ does not exist and then $X^{\prime}=\{b\}$.

First we consider the case $X=X^{\prime}=\{b\}$. Note that in this case, $T^-_v$ covers all vertices of $X\cup Y$. If $a\in Y^{\prime}$, that is, $a\in T^+_u$, or there is an in-arc $a_Ia$ of $a$ in $D\left\langle X\cup Y\right\rangle$ not used in $T^-_v$, then $T^+_u\cup \{br:r\in Y-Y^{\prime}\}$ or $T^+_u\cup \{a_Ia\}\cup\{br:r\in Y-Y^{\prime}-a\}$ extends $T^+_u$, a contradiction. So (i) follows by the assumption.

For the case that $X-b\neq \emptyset$, the first statement of (ii) follows from Lemma \ref{XYnoarc} when we consider  $X-b$ and $Y$. Suppose that there is an arc $wz$ with $w\neq b$ in $D\left\langle X\right\rangle$ which is not in $T^+_u$. Then $(T^+_u,T^-_v)$ can be extended in the following way: $T^+_u\cup \{wr:r\in Y-Y^{\prime}\}$ and $T^-_v\cup \{wz\}\cup\{ra:r\in X-b-w\}$, which contradicts our assumption. So all arcs in $D\left\langle X\right\rangle$ which not covered by $T^+_u$ are out-arcs of $b$ and then (ii) holds. 
\end{proof}

\begin{lem}\label{XYarc}
Let $ab$ be an arc from $Y$ to $X$. Suppose that all arcs between $X$ and $Y$ go from $X$ to $Y$ except for the arc $ab$. If $X^{\prime}=Y^{\prime}=\emptyset$ and $(T^+_u,T^-_v)$ is non-extendable, then one of the following statements holds: 

(i) $X=\{b,x\}, Y=\{a,y\}$, $A(D\left\langle X\right\rangle)=\{bx\}$ and
$A(D\left\langle Y\right\rangle)=\{ya\}$ and either $bx\in T^+_u$ or $ya\in T^-_v$;

(ii) $Y=\{a\}$ and either all out-arcs of $b$ in $D\left\langle X\right\rangle$ are used in $T^+_u$ or all arcs not covered by $T^+_u$ in $D\left\langle X\right\rangle$ are out-arcs of $b$.

(iii) $X=\{b\}$ and either all in-arcs of $a$ in $D\left\langle Y\right\rangle$ are used in $T^-_v$ or all arcs not covered by $T^-_v$ in $D\left\langle Y\right\rangle$ are in-arcs of $a$.
\end{lem}
\begin{proof}
First observe that no arc between $X$ and $Y$ is used in $T^+_u$ and $T^-_v$ as $X^{\prime}=Y^{\prime}=\emptyset$. Suppose that $|X|\geq 3$ and $|Y|\geq 2$, say $x_1,x_2,b\in X$ and $y,a\in Y$. Then $T^+_u\cup \{x_1y,x_2a\}\cup\{br:r\in Y-a-y\}$ extends $T^+_u$ and $T^-_v\cup \{x_1a\}\cup\{ry:r\in X-x_1\}$ extends $T^-_v$, a contradiction. So we may assume that either $|X|\leq 2$ or $|Y|= 1$. By symmetry we have that either $|Y|\leq 2$ or $|X|=1$.

Next we show that if $|Y|\geq 2$ and $D\left\langle X\right\rangle$ has an arc $wz$ with $w\neq b$ not in $T^+_u$, then $(T^+_u,T^-_v)$ is extendable. Let $y$ be a vertex in $Y-a$. Then $(T^+_u,T^-_v)$ can be extended as follows: $T^+_u\cup \{wr:r\in Y\}$ and $T^-_v\cup \{wz\}\cup\{ry:r\in X-w\}$. By assumption, we may assume that either $|Y|= 1$ or all arcs in $D\left\langle X\right\rangle$ which not in $T^+_u$ are out-arcs of $b$. By symmetry,  we have that either $|X|= 1$ or all arcs in $D\left\langle Y\right\rangle$ which not in $T^-_v$ are in-arcs of $a$.

Suppose that $|X|=|Y|=2$, say $X=\{x,b\}$ and $Y=\{y,a\}$. If $x$ dominates $b$, then $xb$ belongs to $T^+_u$ by the argument above and hence $xab\cup xy$ and $T^-_v\cup xby$ extend $T^+_u$ and $T^-_v$, respectively\footnote{Note that in this case the new out-tree does not use all arcs of the old out-tree $T^+_u$.}. So we assume that $D\left\langle X\right\rangle$ consists of arc $bx$. By symmetry, we have $D\left\langle X\right\rangle=ya$. Moreover, either $bx\in T^+_u$ or $ya\in T^-_v$, otherwise, $(bya,bxa)$ extends $(T^+_u,T^-_v)$. This implies that (i) holds.

Now it suffices to consider the case that $|X|=1$ or $|Y|=1$. Suppose that $|Y|=1$. If $D\left\langle X\right\rangle$ has an out-arc $bb_o$ of $b$ and an arc $wz$ with $w\neq b$ such that  none of the arcs $bb_0$ and $wz$ are used in $T^+_u$, then  $(T^+_u,T^-_v)$ can be extended in the following way: $T^+_u\cup \{wa\}$ and $T^-_v\cup \{wz,bb_o\}\cup\{ra:r\in X-b-w\}$, which contradicts our assumption. So (ii) holds and by symmetry (iii) holds if $|X|=1$. 
\end{proof}

%

\section{Arc-disjoint branchings and paths in semicomplete digraphs}\label{sec:B+P}
If $D$ has a good $(u,v)$-pair, then for every vertex $w$ of $D$ it is the case that $D$ has a $(w,v)$-path which is arc-disjoint from some out-branching $B^+_u$ rooted at $u$. It turns out that this partial problem also has a simple and very natural characterization for semicomplete digraphs.

\begin{thm}\label{Buvwiff}
Let $D=(V,A)$ be a semicomplete digraph and let $u,w,v$ be three vertices (not necessarily distinct) such that $D$ has an out-branching rooted at $u$ and a $(w,v)$-path. Then $D$ has an out-branching rooted at $u$ which is arc-disjoint from some $(w,v)$-path if and only if $D$ has a pair of arc-disjoint $(u,z)$- and $(w,v)$-paths for every vertex $z\in V$.
\end{thm}

\begin{proof}
The necessity is trivial. To see the sufficiency, observe that for the case $u=w$, by Lemma \ref{OutbranPath}, we may assume that there exists a partition $V_1,V_2$ of $V$ such that $v\in V_1,u\in V_2$ and $d_D^+(V_2)=1$. Then for any vertex $z\in V_1$, $D$ has  no pair of arc-disjoint $(u,z)$- and  $(w,v)$-paths,  contradicting our assumption. So it suffices to consider the case $u\neq w$.

Let $D^{\prime}$ be an auxiliary digraph obtained from $D$ by adding a new vertex $s$ together with arcs $su,sw$. If there are two arc-disjoint out-branchings with root $s$ in $D^{\prime}$, then it is clear that $D$ has the desired branching and path. Hence by Theorem \ref{edmonds1973}, we may assume that there is a subset $V_1\subseteq V(D^{\prime}) - \{s\}$ such that $d_{D^{\prime}}^-(V_1) \leq 1$. Let $V_2=V-V_1$. Clearly, $s\in V_2$. By the construction and the fact that  $d_{D^{\prime}}^-(V_1) \leq1$, we have either $u\notin V_1$ or $w\notin V_1$. Furthermore, since $D$ has an out-branching rooted at $u$, we have either $u\in V_1,w\in V_2$ or $u,w\in V_2$. 

If  $u\in V_1,w\in V_2$, then the arc $su$ is the only arc entering $V_1$ in  $D^{\prime}$ and hence $V_1$ has in-degree zero in $D$. Since there is a $(w,v)$-path in $D$, the vertex $v$ belongs to $V_2$ and $D\left\langle{}V_2\right\rangle$ contains a
$(w,v)$-path $P$. Now the desired pair can be obtained by taking $P$ and an out-branching in $D$ consisting of an out-branching with root $u$ in $D^{\prime}\left\langle V_1\right\rangle$ and  all arcs $\{ur:r\in V_2-s\}$. So we may assume that both $u$ and $w$ belong to $V_2$ and moreover, $d_{D^{\prime}}^-(V_1) =d_{D}^-(V_1)=1$ as $D$ has an out-branching rooted at $u$. Let $xy$ be the arc entering $D\left\langle V_1\right\rangle$. By the assumption, there is a pair of arc-disjoint $(u,y)$- and $(w,v)$-paths $P_{u,y}$ and $P_{w,v}$
in $D$. Thus we must have $v\in V_2$. Now we construct an out-branching $B_u^+$ in $D$ which is arc-disjoint with $P_{w,v}$ as follows: $B_{u}^+=B_{y,D\left\langle V_1\right\rangle}^+\cup P_{u,y}\cup\{yr:r\in V_2-\{s\}-V(P_{u,y})\}$, where $B_{y,D\left\langle V_1\right\rangle}^+$ exists as $D$ has an out-branching rooted at $u$. This completes the proof.	
\end{proof}
\newpage

\begin{thm}\label{outbranpath}
Let $D$ be a semicomplete digraph and let $u,w,v$ be three vertices (not necessarily distinct) such that $D$ contains an out-branching rooted at $u$ and a $(w,v)$-path. Then $D$ has an out-branching with root $u$ which is arc-disjoint from some $(w,v)$-path unless one of the following statements holds.

(i) $Out(D)=\{u\}=\{w\}$ and \JBJ{$u$ is the only in-neighbour of }$v$ in $D$. In particular, $(D,u,w,v)$ is of type $A$ with $V_1=\{v\}$ in Definition \ref{def1}.

(ii) The vertices $u,w,v$ belong to the same component of $D$, i.e., $D\left\langle Out(D)\right\rangle$, and $(D\left\langle Out(D)\right\rangle,u,w,v)$ is one of the types in Definition \ref{def1}.
\end{thm} 

\begin{proof} 
	Observe that $u\in Out(D)$ as there is an out-branching rooted at $u$ in $D$. Then the theorem follows by Theorem \ref{Buvwiff} and  Theorem \ref{goodpair-paths}. It should be noted that if $(D\left\langle Out(D)\right\rangle,w,u,z)$ is one of the types in Definition \ref{def1} and $v\in V_1$, then $(D\left\langle Out(D)\right\rangle,u,w,v)$ is of type $A$.
%
\end{proof}

\section{Arc-disjoint in- and out-branchings in semicomplete digraphs}\label{sec:proof}

\noindent{}Now we are ready to prove Theorem \ref{thm:SDbranchchar}.
We first state the following result that will be used in the proof. An arc $xy$ of a strong digraph $D$ is a {\bf cut-arc} if $D\setminus xy$ is not strong.

\begin{thm}\label{thmOI}
Let $D$ be a strong semicomplete digraph of order at least 4 with a cut-arc $xy$ and let $u \in Out(D-xy),v \in In(D-xy)$ be two vertices. Suppose that $D$ contains no good $(u,v)$-pair. Then either $D$ is isomorphic to one of the digraphs shown in Figure \ref{fig4} (d)-(f) or $D-xy$ has exactly two strong components and one of the following statements holds. 

(i) $In(D-xy)=\{v\}=\{x\}$ and $d_{D}^+(y)=1$. Say $N_D^{+}(y)=\{z\}$. There is no $(u,z)$-path in $D-yz$.

(ii) $Out(D-xy)=\{u\}=\{y\}$ and $d_{D}^-(x)=1$. Say $N_D^{-}(x)=\{z\}$. There is no $(z,v)$-path in $D-zx$.
\end{thm}

\noindent{} Note that if (i) holds, then  we have $z\in V_1$, $V_{l-1}=\{x\}=\{v\},V_l=\{y\}$ and $u\notin V_1\cup V_{l-1}$ (possibly $u=y$) where $D_1,\ldots,D_l~(l\geq 3)$ is acyclic ordering of the strong components of $D-yz$ and $V_i=V(D_i)$.\\

\begin{proof}  Let $Y=In(D-xy)$ and $X=V(D)-Y$. As $u$ belongs to $Out(D-xy)=Out(D\left\langle X\right\rangle)$ and $v\in Y= In(D-xy)$, $D$ has an  out-branching rooted at $u$ in $D\left\langle X\right\rangle$ and an in-branching rooted at $v$ in $D\left\langle Y\right\rangle$. Let $T^+_u$ and $T^-_v$ be any such out- and in-branchings, respectively. As $D\left\langle Y\right\rangle$ is strong, by symmetry, we may assume that the second statement of Lemma \ref{XYarc} holds. That is, $Y=\{x\}=\{v\}$ and either all out-arcs of $y$ in $D\left\langle X\right\rangle$ are used in $T^+_u$ or all arcs of  $D\left\langle X\right\rangle$ which are not used by $T^+_u$  are adjacent to $y$.

Suppose that $d_D^+(y)\geq 2$. Let $T^+_u$ be a hamitonian path starting at $u$ in $D\left\langle X\right\rangle$. Then there is an out-arc of $y$ not used in $T^+_u$ and hence, by the remark above, all arcs not covered by $T^+_u$ in $D\left\langle X\right\rangle$ are adjacent to $y$. It follows by Lemma \ref{adjacentsomevex} and the fact $|V(D)|\geq 4$ that $|X|=3$. Suppose that $T^+_u=u_1u_2u_3$ with $u_1=u$ is the hamiltonian path in $D\left\langle X\right\rangle$. Then $y\in\{u_1,u_3\}$. First consider the case $y=u_1=u$. If $u_1u_3\notin A(D)$, then $y$ dominates $x$ as $d_D^+(y)\geq 2$. Then $T^+_u\cup \{u_3v\}$ and $u_3yv\cup \{u_2v\}$ is a good $(u,v)$-pair, which contradicts our assumption. So $u_1u_3\in A(D)$. If $u_2$ dominates $u_1$, then $T^+_u\cup \{u_2v\}$ and $u_2u_1u_3v$ form a good $(u,v)$-pair, a contradiction again. Then $D$ is isomorphic to the digraph shown in Figure \ref{fig4} (d) if $u_3u_1\notin A(D)$ (resp., in Figure \ref{fig4} (f) if $u_3u_1\in A(D)$.  For the case that $y=u_3$, the vertex $u_3$ dominates $u_1$ as $d_D^+(y)\geq 2$ and then $u_1u_2vu_3$ and $u_2u_3u_1v$ form a good $(u,v)$-pair,  a contradiction again.

\JBJ{It remains }to consider the case $d_D^+(y)=1$. Say $N_D^+(y)=\{z\}$. This means that $yz$ is a cut-arc of $D$. Let $V_1,\ldots,V_l$ be the acyclic ordering of the strong components of $D-yz$. It follows by $d_D^+(y)=1$ that $|V_l|=|\{y\}|=1$. Since there \JBJ{is} only one $(x,y)$-path in $D$ (recall that $xy$ is a cut-arc), we have  $V_{l-1}=\{x\}=\{v\}$. As $u \in Out(D-xy)$ and $v \in In(D-xy)$, we have $u\neq v$ and then $u\notin V_{l-1}$.

Note that if $u\notin V_1$, then  the statement (i) holds so we may assume that $u\in V_1$. Next we claim that there is a good $(u,v)$-pair, which contradicts our assumption. Recall that $X=V_{1}\cup\cdots\cup V_{l-2}\cup \{y\}$. As $|V(D)|\geq 4$ we have that $|V_{1}\cup\cdots\cup V_{l-2}|\geq 2$. Moreover, there is a spanning out-tree (that is, an out-branching) $T^+_u$ in $D\left\langle X\right\rangle$ such that $yz$ does not belong to $T^+_u$. By the argument in the first paragraph of the proof, all arcs not covered by $T^+_u$ in $D\left\langle X\right\rangle$ are adjacent to $y$. Thus all arcs in $D\left\langle V_{1}\cup\cdots\cup V_{l-2}\right\rangle$ are used in $T^+_u$, which means that $|V_{1}\cup\cdots\cup V_{l-2}|\leq 2$. It follows by $|V(D)|\geq 4$ that equality holds, say $D\left\langle V_{1}\cup\cdots\cup V_{l-2}\right\rangle=uw$, then $uwvy$ and $wyuv$ form the desired pair, which completes the proof.	
\end{proof}

\begin{lem}\label{counterexample}
Let $D$ be a semicomplete digraph and $u,v$ two distinct vertices. If $D$ is isomorphic to one of digraphs shown in Figure \ref{fig4}, then there is no good $(u,v)$-pair in $D$.
\end{lem}
\begin{proof}
	If $D$ has a good $(u,v)$-pair, then the size of $D$ must be at least $2(|V(D)|-1)$, moreover,  if $v$ dominates $u$, then  $|A(D)|\geq 2|V(D)|-1$ since  the arc $vu$ cannot be used in any $(u,v)$-pair. This shows that  $D$ has no good $(u,v)$-pair if it is isomorphic to one of digraphs in Figure \ref{fig4} (a)-(d).
	
	For the case that $D$ is isomorphic to the digraph in Figure \ref{fig4} (e), suppose that $D$ has a good $(u,v)$-pair ($B_u^+,B_v^-$). Then every arc except for $vu$ either belongs to $B_u^+$ or $B_v^-$ as $|A(D-vu)|=2|V(D)|-2$. Let $V(D)-\{u,v\}=\{w,z\}$ such that $w$ dominates $z$. The only out-arc $zv$ of $z$ must belong to $B_{v}^-$ and $vw$ must belong to $B_{u}^+$. By the definitions of out- and in-branchings, we have $wv,uw\in A(B_v^-)$ and then $B_u^+=uz\cup vwz$. However, $B_u^+$ is not an out-branching rooted at $u$, a contradiction. By a similar argument, $D$ has no good $(u,v)$-pair when it is not isomorphic to the digraph shown in Figure \ref{fig4} (f).
\end{proof}

\noindent{}{\bf Proof of Theorem \ref{thm:SDbranchchar}:}
\begin{proof}
  Observe that if there is a good $(u,v)$-pair, then clearly there are arc-disjoint $(u,z)$- and $(w,v)$-paths for every choice of vertices $z$ and $w$.  Hence the necessity follows by Lemma \ref{counterexample}. Next we show the sufficiency. If $u=v$, then it follows from Theorem  \ref{sABC} that the desired branchings exist since if there was an arc $e=pq$ as in the theorem, then $D$ would have no pair of arc-disjoint $(u,q)$- and $(p,x)$-paths. Hence we can assume that $u\neq v$.

  Suppose first that $D$ is non-strong. Since there is a $(u,z)$-path and a $(w,v)$-path for any $z,w$, we have that $u$ belongs to $Out(D)$ and $v$ belongs to $In(D)$. By Lemma \ref{nonstrong-OI}, we may assume that $|V(D)|\leq 3$. Since $D$ is not isomorphic to one of digraphs shown in Figure \ref{fig4} (a)-(b), either $|Out(D)|=2$ or $|In(D)|=2$ and then there is a good $(u,v)$-pair in $D$.

  It remains to consider the case that $D$ is strong. Further, by Theorem \ref{thm:SD2arcstrong}, we may assume that $D$ is not 2-arc-strong, implying that it has a cut-arc. Let $xy$ be a cut-arc of $D$ and let $U_1,\ldots,U_l$ $(l\geq 2)$ be the acyclic ordering of the strong components of $D-xy$. Note that $y\in U_1$ and $x\in U_l$. Suppose that $u$ belongs to $U_i$ and $v$ belongs to $U_j$. Since there exist arc-disjoint $(u,y)$- and $(x,v)$-paths in $D$,  either $u$ belongs to $U_1$ or $v$ belongs to $U_l$ (or both), that is, either $i=1$ or $j=l$.
	
Note that $|V(D)|\geq 3$ as it has two arc-disjoint $(u,v)$-paths. If $|V(D)|=3$, say $V(D)=\{u,v,w\}$, then $u$ dominates $w,v$ and $w$ dominates $v$. It can be checked easily that there is a good $(u,v)$-pair if $v$ dominates $w$ or $w$ dominates $u$. Hence we can assume that $v$ dominates $u$ and then $D$ is isomorphic to the digraph shown in Figure \ref{fig4} (c), contradicting our assumption. Therefore, we may assume that $|V(D)|\geq 4$. 
	 
	Suppose first that $u\in U_1$ and $v\in U_l$, that is, $i=1,j=l$. By Theorem \ref{thmOI} we are done, unless either (i) or (ii) in the theorem holds. By symmetry, we may assume that the statement (i) of Theorem \ref{thmOI} holds. However, then there is no pair of arc-disjoint $(u,z)$- and $(y,v)$-paths in $D$, which contradicts our assumption.

        Consider next the case $i=1$ and $j\neq l$. Let $X=U_1\cup\cdots\cup U_j$ and $Y=V(D)-X$.
        If $D\left\langle X\right\rangle$ does not have an out-branching rooted at $u$ which is arc-disjoint from some $(y,v)$-path, then by Lemma \ref{outbranpath}, $(D,u,y,v)$ is one of the types in Definition \ref{def1}. It follows by Lemma \ref{type--nopaths} that $D$ has no arc-disjoint $(u,z)$- and $(y,v)$-paths for any vertex $z\in V_1$, contradicting our assumption. So we may assume that $D\left\langle X\right\rangle$ has an out-branching $T^+_u$ rooted at $u$ and a $(y,v)$-path $P_{y,v}$ which are arc-disjoint. 
	
	Let $P_{x}$ be a hamiltonian path ending in $x$ in $D\left\langle Y\right\rangle$. Such path $P_{x}$ exists as $x$ belongs to $In(D-xy)$. Let $T^-_v=P_x\cup\{xy\}\cup P_{y,v}$. Clearly, all vertices of $P_{y,v}$ are covered by $T^-_v$, that is, $X^{\prime}=V(P_{y,v})$. Now Lemma \ref{XYblue} implies that we are done unless $|X^{\prime}|=|V(P_{y,v})|=1$. This means that $v=y\in Out(D-xy)$ and then $X=Out(D-xy)$. Moreover, $|X|\geq 2$ as $u\neq v$. Hence, we may assume that the second statement of  Lemma \ref{XYblue} holds. That is, $|(X-y)\cup Y|\leq 3$ and $D\left\langle X-y\right\rangle$ (resp., $D\left\langle Y\right\rangle$) is a single vertex or an arc covered by $T^+_u$ (resp., by $T^-_v$). Moreover, all arcs in $D\left\langle X\right\rangle$ which not covered by $T^+_u$ are out-arcs of $y$.
	
	On the other hand, as $|V(D)|\geq 4$, we have $|(X-y)\cup Y|\geq 3$ and then the equality holds. Suppose that $|Y|=2$, say $Y=\{x,x^{\prime}\}$, then $|X|=2$ and $X=\{u,v\}$ with $v=y$ as $u\neq v$. In this case,  $uvx^{\prime}\cup ux$ and $ux^{\prime}xv$ form a good $(u,v)$-pair. Thus it remains to consider the case that $|Y|=1$ and $|X|=3$, say $X=\{u,v,w\}$ with $v=y$. Recall that $D\left\langle X\right\rangle$ is strong as $X=Out(D-xy)$. Since $D\left\langle X-y\right\rangle$ is an arc covered by $T^+_u$ (as we are in Case (ii) of Lemma \ref{XYblue}) and all arcs in $D\left\langle X\right\rangle$ which are  not used  by $T^+_u$ are out-arcs of $y$ (i.e., $v$), we may assume that $uwvu$ is a hamitonian cycle in $D\left\langle X\right\rangle$ and $T^+_u=uwv$. Then $D$ is isomorphic to the digraph in Figure \ref{fig4} (d) or (e),
        contradicting our assumption.
	
	Finally assume that that $i\neq 1$ and $j=l$. By applying similar arguments as above  to the digraph obtained from $D$ by reversing all arcs, we either find that $D$ has the desired branchings or it would be isomorphic to one of  the digraphs in Figure \ref{fig4} (d) or (f), a contradiction again. This completes the proof of Theorem \ref{thm:SDbranchchar}. \end{proof}

        All our arguments leading to the proof of Theorem \ref{thm:SDbranchchar} are constructive and can be converted to polynomial algorithms. Hence we have the following corollary.

        \begin{coro}
          There exists a polynomial algorithm which given a semicomplete digraph $D=(V,A)$ and vertices $u,v$ of $D$ either constructs a good $(u,v)$-pair or produces a certificate that $D$  has no such pair. 
          \end{coro}

        The following equivalent structural characterization of semicomplete digraphs with good $(u,v)$-pairs is often  easier to use when one wishes to study digraphs that are more general than semicomplete digraphs.
        In particular we use it in  \cite{bangQTDbranchpaper} to study good $(u,v)$-pairs in so-called semicomplete compositions.

\begin{thm}\label{SDgoodpair}
Let $D$ be a semicomplete digraph and $u,v$ be arbitrary chosen vertices (possibly $u=v$). Then $D$ has a good $(u,v)$-pair if and only if $(D,u,v)$ satisfies none of the following conditions.

(i) 
$D$ is isomorphic to one of the digraphs in Figure \ref{fig4}.

(ii) $D$ is non-strong and either $u$ is not in the initial component of $D$ or  $v$ is not in the terminal component of $D$.

(iii) $D$ is strong and there exists an arc $e\in A(D)$ such that $u$ is not in the initial component of $D-e$ and $v$ is not in the terminal component of $D-e$.

(iv) $D$ is strong and there exists a partition $V_1,\ldots,V_{2\alpha+3}$ of $V(D)$ for some $\alpha\geq 1$ such that $v\in V_2, u\in V_{2\alpha+2}$ and all arcs between $V_i$ and $V_j$ with $i<j$ from $V_i$ to $V_j$ with the following exceptions. There exists precisely one arc from $V_{i+2}$ to $V_{i}$ for all $i\in[2\alpha+1]$  and it goes from the terminal component of $D\left\langle  V_{i+2}\right\rangle$ to the initial component of $D\left\langle  V_i\right\rangle$.
\end{thm}

\begin{proof}
We first prove the necessity.  Lemma \ref{counterexample} shows that $D$ has no such pair if (i) holds. It is not difficult to check that when (ii) or (iii) holds, there is no good $(u,v)$-pair in $D$ as no such pair can cover the vertices in the initial and terminal components in the same time. If (iv) holds, let $z$ and $w$ be two vertices of $V_1$ and $V_{2\alpha+3}$, respectively. Then there is no pair of arc-disjoint $(u,z)$- and $(w,v)$-paths (due to Lemma \ref{type--nopaths}). By Theorem \ref{thm:SDbranchchar}, $D$ has no good $(u,v)$-pair.

Now suppose that $(D,u,v)$ satisfies none of the conditions (i)-(iv). We prove that then there exists a good $(u,v)$-pair. If $D$ is non-strong,  since $(D,u,v)$ does not satisfy  condition (ii), we have $u\in Out(D)$ and $v\in In(D)$. From Lemma \ref{nonstrong-OI}, we may assume that $D$ has order at most three. Since $D$ is not isomorphic to one of digraphs shown in Figure \ref{fig4} (a)-(b), either $|Out(D)|=2$ or $|In(D)|=2$ and then there is a good $(u,v)$-pair in $D$. 

For the case that $D$ is strong, suppose for contradiction that there is no good $(u,v)$-pair. Then by Theorem \ref{thm:SDbranchchar}, there exist $z,w$ such that there is no pair of arc-disjoint $(u,z)$-  and $(w,v)$-paths in $D$. Clearly, $D$ has a $(u,z)$-path and a $(w,v)$-path as $D$ is strong.  By Theorem \ref{goodpair-paths},  either $(D,u,w,v)$ is one of the types in Definition \ref{def1}, or $(D,w,u,z)$ is one of the types in Definition \ref{def1} and the vertex $v$ belongs to $V_1$. It is not difficult to check that $(D,u,v)$ satisfies condition (iii) or (iv), which contradicts our assumption. This completes the proof.
\end{proof}

\noindent{}{\bf Conflict of interest statement}\\
There are no sources of conflict of interest regarding this paper.\\

\noindent{}{\bf Data availability statement}\\
Data sharing is not applicable to this article as no datasets were generated or analysed during the current study.

\noindent{}{\bf Acknowledgement:} Financial support from the Independent Research Fund Denmark under grant DFF-7014-00037B is gratefully acknowledged. The second author was supported by China Scholarship Council (CSC) No. 202106220108.


\end{document}